\documentclass{amsart}
\usepackage{amsthm}
\usepackage{amssymb}
\usepackage{amsrefs}
\usepackage{enumerate}
\usepackage{comment}

\numberwithin{equation}{section}
\theoremstyle{plain}
\newtheorem{thm}[equation]{Theorem}

\newtheorem{prop}[equation]{Proposition}
\newtheorem{lem}[equation]{Lemma}
\newtheorem{hyp}[equation]{Hypothesis}

\theoremstyle{definition}
\newtheorem{defn}[equation]{Definition}
\newtheorem{rem}[equation]{Remark}

\DeclareMathOperator{\im}{im}
\DeclareMathOperator{\diag}{diag}
\DeclareMathOperator{\GL}{GL}
\DeclareMathOperator{\PGL}{PGL}
\DeclareMathOperator{\SL}{SL}
\DeclareMathOperator{\SU}{SU}
\DeclareMathOperator{\Sp}{Sp}
\DeclareMathOperator{\U}{U}
\DeclareMathOperator{\SO}{SO}
\DeclareMathOperator{\Spin}{Spin}
\DeclareMathOperator{\PU}{PU}
\DeclareMathOperator{\Gal}{Gal}
\DeclareMathOperator{\Hom}{Hom}
\DeclareMathOperator{\Ind}{Ind}
\DeclareMathOperator{\cInd}{c-Ind}
\DeclareMathOperator{\Rest}{R}
\DeclareMathOperator{\coker}{coker}
\newcommand{\ff}{\mathfrak{f}}
\newcommand{\sfG}{\mathsf{G}}
\newcommand{\sfH}{\mathsf{H}}
\newcommand{\sfS}{\mathsf{S}}
\newcommand{\TT}{\mathcal{T}}

\newcommand{\mult}{^\times}
\newcommand{\bdd}{_{\mathrm{b}}}
\newcommand{\odd}{_{\mathrm{odd}}}

\newcommand{\der}{{\mathrm{der}}}

\begin{document}

\title{Self-dual cuspidal representations}

\author{Jeffrey D. Adler}
\thanks{The first-named author
was partially supported by
the American University College of Arts and Sciences
Faculty Research Fund.}
\address{
Department of Mathematics and Statistics \\
American University \\
4400 Massachusetts Ave NW \\
Washington, DC  20016-8050 \\
USA
}
\email{jadler@american.edu}

\author{Manish Mishra}
\thanks{The second-named author was partially supported 
by SERB MATRICS and SERB ECR grants}
\address{
Department of Mathematics \\
Indian Institute for Science Education and Research \\
Dr.\ Homi Bhabha Road, Pashan \\
Pune 411 008 \\
India 
}
\email{manish@iiserpune.ac.in}

\date{2019 October 28}

\begin{abstract}
Let $G$ be a connected reductive group
over a finite field $\ff$ of order $q$.
When $q\leq 5$, we make further assumptions on $G$.
Then
we determine precisely when $G(\ff)$ admits irreducible, cuspidal
representations that are self-dual, of Deligne-Lusztig type,
or both.
Finally, we outline some consequences for the existence
of self-dual supercuspidal representations of reductive $p$-adic groups.
\end{abstract}

\maketitle

\section{Introduction }
\label{sec:intro}
Let $G$ denote a connected reductive $F$-group,
where $F$ is either a finite
or a local nonarchimedean field.
A representation $\pi$ of $G(F)$
is called \emph{self-dual}
(or sometimes \emph{self-contragredient} when $F$ is local)
if $\pi$ is isomorphic to its (smooth) dual $\pi^\vee$.
In this article, we give necessary and sufficient conditions
for the existence of complex, irreducible, self-dual
(super)cuspidal representations of such groups.

When $F$ is finite, we also determine when one can find
such representations that are of Deligne-Lusztig type.
Using that result, we determine, in the case where $F$ is local,
when one can find such representations which are \emph{regular}
and have depth zero.
The latter are those representations that arise from suitable Deligne-Lusztig cuspidal
representations of reductive quotients of parahoric subgroups of $G(F)$ (see Definition \ref{def:reg-dep0}).

The existence or non-existence of self-dual representations for specific groups
has been studied in several works.
Suppose $F$ is local.
When $G$ is a general linear group,
the first-named author gave necessary and sufficient conditions
for the existence of \emph{tame} self-dual supercuspidal 
representations \cite{adler:self-contra}.
Here, ``tame'' refers to the supercuspidals
constructed by Howe \cite{howe77}.
When the residue characteristic $p$ of $F$ is even,
Bushnell-Henniart showed the existence of self-dual representations
for linear groups and division algebras of odd degree \cite{BH2011}. 
For division algebras of odd degree when $p$ is odd,
D. Prasad showed the non-existence 
of self-dual representations of dimension greater than $1$ \cite{dprasad:div-alg}.

Some of our results require us to impose certain hypotheses on $G$.
These hypotheses disallow $G$ to have 
certain small-rank factors of type ${}^2A_k$, $k\leq 4$ (see Terminology)
when the field (in case $F$ is finite)
or the residue field (in case $F$ is local) is of cardinality $\leq 5$. 

Our main results are as follows.
\begin{enumerate}[(A)]
\item
(Theorem \ref{thm:existence-sd-finite})
Let $\ff$ denote a finite field of order $q$.
Let $G$ denote a connected reductive $\ff$-group.
Then $G(\ff)$ admits irreducible, cuspidal representations.
If $G$ satisfies Hypothesis \ref{hyp:SU-restrictions}(a),
then $G(\ff)$ admits irreducible, cuspidal, \emph{Deligne-Lusztig}
representations.
If $G$ also satisfies Hypothesis \ref{hyp:SU-restrictions}{(b)},
then 
the following are equivalent.
\begin{enumerate}[(i)]
\item
$G(\ff)$ admits irreducible, \emph{self-dual}, cuspidal representations.
\item
$G(\ff)$ admits irreducible, self-dual, cuspidal,
\emph{Deligne-Lusztig} representations.
\item
$G$ has no simple factor of type $A_n$ for any even $n$
(see ``Terminology'' below).
\end{enumerate}
\item
Let $F$ denote a non-archimedean local field with residue field
$\ff$ of order $q$.
Let $G$ denote a connected reductive $F$-group.
If $G$ satisfies Hypothesis \ref{hyp:SU-restrictions-padic}(a),
then
(Theorem \ref{thm:self-dual-sc})
$G(F)$ admits regular, depth-zero, supercuspidal representations.
Moreover,
if $G$ also satisfies Hypotheses \ref{hyp:SU-restrictions-padic}(b)
and \ref{hyp:S-decomp},
then
$G(F)$ admits irreducible, \emph{self-dual}, regular,
depth-zero supercuspidal representations.
Moreover
(combining this result with Proposition \ref{prop:no-self-dual}),
if the residue characteristic $p$ of $\ff$ is odd,
then the following are equivalent:
\begin{enumerate}[(i)]
\item
$G(F)$ admits irreducible, self-dual, supercuspidal
representations;
\item
$G(F)$ admits irreducible, self-dual, regular,
\emph{depth-zero} supercuspidal representations;
\item
$G$ has no $F$-almost-simple factor of type $A_n$ for any even $n$.
\end{enumerate}
Finally, (Theorem \ref{thm:self-dual-sc-even})
if $p=2$, and either $q\neq 2$ or $G$ has no factor of type
$^2A_3$ or $^2A_4$,
then $G(F)$ admits irreducible, self-dual supercuspidal representations.
\end{enumerate}

In the course of the proofs of our theorems, we show the existence
of cuspidal representations of all connected reductive $\ff$-groups
(Theorem \ref{thm:existence-cuspidal-finite}),
and thus depth-zero supercuspidal representations of all 
connected reductive $F$-groups
(Proposition \ref{prop:existence-sc-depth0}),
without any restriction on $\ff$ or $F$.
This result is folklore, and can be inferred from 
\cite{dat-orlik-rapoport:period-domains}*{Prop.\ 7.1.4}
using some facts about dual groups.
(The proof in \emph{loc.~cit.} omitted the case of the group $G_2(2)$, but the
result is nonetheless true, as can be seen below.) When $F$ has characteristic zero,
the existence of supercuspidals was also proved by Beuzart-Plessis
\cite{beuzart-plessis:sc} using methods of harmonic analysis, bypassing
questions about finite groups.

Our proofs of Theorems (A) and (B) are mostly uniform, except for the fact
that certain unitary and orthogonal groups require special
handling, as do several other groups when $q=2$.

We thank Tasho Kaletha and Loren Spice for helpful conversations;
Dipendra Prasad for pointing out a serious error in an earlier draft;
and an anonymous referee for many helpful suggestions to improve
clarity.
Our proof of Proposition \ref{prop:coxeter} benefits from an idea in the
proof of \cite{dat-orlik-rapoport:period-domains}*{Prop.\ 7.1.4},
which shows the existence of semisimple elliptic elements
in an arbitrary finite reductive group. We also learned useful things from  
an unpublished note of Arno Kret on the existence of cuspidal representations.

\subsection*{Terminology}
Let $F$ be any finite or nonarchimedean local field.
If $G$ is $F$-almost-simple,
then $G$ is isogenous to $\Rest_{E/F} H$ for some finite extension $E/F$
and some absolutely almost-simple group $H$.
If we say that $G$ has a certain ``type'', then we are specifying
both the absolute root system of $H$ and the $*$-action on this root system
of the absolute Galois group of $E$.
The ``type'' will sometimes include an indication of the order of $E$,
if $E$ is finite.
Thus, for example,
``$A_n$'' refers to a group that is $F$-isogenous to 
an inner form of
$\Rest_{E/F}\SL_{n+1}$; 
``$^2A_n$'' refers to a group that is $F$-isogenous to 
an inner form of
$\Rest_{E/F}\SU_{n+1}$;
and
``$^2A_n(q)$'' 
refers more specifically to a group that is $F$-isogenous to
$\Rest_{E/F}\SU_{n+1}$,
where $E$ has order $q$.
Whenever $F$ is local, a twisted type (e.g., $^2 E_6$) will by default
refer to a group that splits over an \emph{unramified} extension of $E$.

\section{Cuspidals arising from elliptic tori}
\label{sec:DL-reps}

Let $G$ be a connected reductive group defined over a finite field $\ff$
of order $q$.
Let $\sigma_{\ff}$ denote the Frobenius endomorphism.
Let $B_{0}$ be a Borel $\ff$-subgroup of $G$
containing a maximally split maximal $\ff$-torus $T_{0}$ of $G$.
Let $\omega$ be
a $\sigma_\ff$-elliptic element in the absolute Weyl group $W=W(G,T_{0})$ and
let $T=T_{\omega}$ be the corresponding elliptic torus.
Note that $T$ depends only on the $\sigma_\ff$-twisted conjugacy class
of $\omega$
in $W$.
The Weyl group $W(G,T)(\ff)$ of $T$ is the $\sigma_{\ff}$-centraliser $\Omega$
of $\omega$ in $W$ \cite{carter:finite}*{Prop.\ 3.3.6}.
There is an $\Omega$-equivariant
isomorphism \cite{carter:finite}*{Prop.\ 3.2.3 and Prop.\ 3.3.4}
\begin{equation}
\label{eqn:kret}
L=L_\omega:=
\frac{X}{(\sigma_{\ff}\omega^{-1}-1)X}
\overset{\sim}{\longrightarrow}
\Hom(T(\ff),\mathbb{C}^{\times}).
\end{equation}
Here $X=X(T_{0})$ denotes the character lattice of $T_{0}$.

A complex character $\chi$ of $T(\ff)$ in general position
gives rise to a Deligne-Lusztig representation $\pi(T,\chi)$ of $G(\ff)$ which
is irreducible and cuspidal. The representation $\pi(T,\chi)$ is
self-dual if and only if the pair $(T,\chi)$ is $\Omega$-conjugate
to $(T,\chi^{-1})$ \cite{digne-michel:reps-book}*{Prop.\ 11.4}.
We will call such a character $\chi$ \emph{conjugate self-dual}.
We will call an element in $L$ 
\emph{conjugate self-dual} (resp.\ \emph{in general position})
if its inverse image under the isomorphism
in \eqref{eqn:kret}
is conjugate self-dual (resp.\ in general position). 

Thus,
to prove the existence of irreducible
self-dual Deligne-Lusztig cuspidal representations
of $G(\ff)$,
it is sufficient prove the existence of conjugate self-dual elements
in $L$ that are in general position. 

We first consider the existence of such elements in the special case
where $T$ is a Coxeter torus.

\begin{prop}
\label{prop:coxeter}
Suppose that $G$ is absolutely almost simple,
$t$ is the degree of the splitting field of $G$,
$T$ is the Coxeter torus in $G$,
and $h$ is the Coxeter number of $G$.
If $h/t$ is odd, then $T(\ff)$ has no conjugate self-dual characters
in general position.
Moreover, suppose that 
$G$ does not have type $^2A_2(2)$ or $G_2(2)$.
Then the following hold:
\begin{itemize}
\item
The group $T(\ff)$ has a character that is in general position.
\item
If $h\neq 2$, then we can choose such a character to have 
order $\ell$, where $\ell$ is a prime such that the multiplicative
order of $q$ mod $\ell$ is $h$.
\item
If $h/t$ is even,
then $T(\ff)$ has such a character that is also conjugate self-dual.
\end{itemize}
\end{prop}

\begin{proof}
Let $\omega$ be a
$\sigma_\ff$-Coxeter element of $W$.
The endomorphism $\sigma_\ff$ of $X$ is of the form 
$\sigma_{0} \cdot q$, where $\sigma_{0}$ is a finite-order automorphism of $X$.
Write $w:=\sigma_{0}\omega^{-1}$ and let $t$ denote the order of $\sigma_0$.
Then
$\Omega$ is a cyclic group generated by
$(w^{-1})^t$ \cite{springer:regular-elements}*{Theorem 7.6(v)}
and $w^{-1}$ acts on the abelian group $L$ by multiplication by $q$. 

Suppose that $h/t$ is odd.
Let $u\in L$ be conjugate self-dual and in general position.
If $u=-u$, then $2u=0$.
Therefore, $qu=u$ if $q$ is odd and $qu=0$ if $q$ even.
In either case, this contradicts $u$ being in general position.
So $(q^t)^{\alpha}u=-u$ for some $0<\alpha<h/t$
and therefore $((q^t)^{2\alpha}-1)u=0$.
But then $2\alpha=h/t$ since $u$ is in general position.
But this contradicts the fact that $h/t$ is odd. 

Now suppose that $h=2$. 
Then $X$ has rank one and $L=\frac{X}{(q+1)X}$.
In this case, a generator of $L$ is a conjugate self-dual
element in general position.

Now suppose that $h \neq 2$ and $(q,h) \neq (2,6)$.
Then by \cite{birkhoff-vandiver}*{Theorem V},
there exists a prime $\ell$ such that the multiplicative
order of $q$ mod $\ell$ is $h$.
Let $\mathrm{ch}_{w}$
denote the characteristic polynomial of the action of $w$ on $X$,
and let $r$ denote the rank of $X$.
Then the order $|L|$ is
$|\det(wq-1)|=|q^{-r}\det(w-q^{-1})|=|q^{-r}\mathrm{ch}_{w}(q^{-1})|
=|\mathrm{ch}_{w}(q)|$.
The last equality follows because $\mathrm{ch}_{w}$
is a product of cyclotomic polynomials,
and thus its sequence of coefficients is symmetric.

Let $\Phi_{h}$ denote the $h^{\mathrm{th}}$ cyclotomic polynomial.
Then $\ell$ divides $\Phi_{h}(q)$.
By \cite{springer:regular-elements}*{Theorem 7.6(ii)},
$\Phi_{h}\mid\mathrm{ch}_{w}(q)$ and therefore $\ell\mid\mathrm{ch}_{w}(q)$.
Therefore $L$ has a cyclic subgroup $C$ of order $\ell$.
Let $v$ be a generator of $C$.
Then $v$ is in general position.

Suppose $h/t$ is even.
Since $q$ has order $h$ mod $\ell$, it follows that $\ell\nmid(q^{h/2}-1)$.
Therefore $\ell\mid(q^{h/2}+1)$. Thus $(q^t)^{(h/t)/2}$ acts by $-1$ on $v$
and therefore $v$ is conjugate self-dual.

It remains to handle the cases where $(q,h)=(2,6)$.
From 
\cite{humphreys:reflection}*{\S3.18, Table 2}
and
\cite{springer:regular-elements}*{Table 10, page 184},
$G$ has one of the following types:
$A_5$, $C_3$, $D_4$, $G_2$,
${}^2A_2$, ${}^2A_3$.
By hypothesis, $G$ does not have type $^2A_2$ or $G_2$,
and we consider each of the other cases in turn.
From Lemma \ref{lem:trivial-iso}, we may replace $G$ by
any isogenous group.
In each case,
it will be sufficient to find a cyclic subgroup of
$L$ (equivalently, of $T(\ff)$)
of order $9$.
For let $v$ be a generator of such a subgroup.
Then $2^3v= -v$, and $2^i v \neq v$ for all $0<i<6$,
so $v$ is in general position.
Moreover, $v$ is conjugate self-dual if $h/t$ is even, i.e.,
$G$ does not have type $^2A_3$.
\begin{description}
\item[Type $A_5$]
The group $T(\ff)$ is cyclic of order $63$,
so it contains a cyclic subgroup of order $9$.
\item[Type $C_3$]
The group $T(\ff)$ is isomorphic to the kernel of $N_{E/K}$,
where $K/\ff$ is a cubic extension and $E/K$ is a quadratic
extension.
Thus, $T(\ff) \cong E\mult / K\mult$, a cyclic group of order $9$.
\item[Type $D_4$]
The group $T(\ff)$ is isomorphic to 
$\mathbb{Z}/3\mathbb{Z}\times \mathbb{Z}/9\mathbb{Z}$,
so it contains a cyclic subgroup of order $9$.
\item[Type $^2A_3$]
The group $T(\ff)$ is cyclic of order $q^3+1=9$.
\qedhere
\end{description}
\end{proof}

\begin{lem}
\label{lem:coprime}
Let $G$ and $\ell$ be as in Proposition \ref{prop:coxeter}
and let $Z$ denote the 
center of $G$.
Assume that $h \neq 2$ and $(q,h) \neq (2,6)$.
Then $\ell$ is co-prime to $|Z|$.
\end{lem}

\begin{proof}
Assume that the absolute root system of $G$ is of type
other than $A_{n-1}$ or $E_6$.
Then $Z$ is a $2$-group. 
Since the order of $q \bmod \ell$ is $h>2$,
we cannot have $\ell = 2$.

Now suppose $G$ is of type $E_6$ or ${}^2E_6$.
Then $h=12$ or $18$ and $|Z|$ divides $3$.
Suppose $\ell=3$.
Then $q^2 \bmod \ell$ is $0$ or $1$,
contradicting that the order of $q$ is $h$.

If $G$ is of type $A_{n-1}$, then $|Z|$ divides $h=n$.
If $\ell$ divides $n$,
then $q^{n/\ell}\equiv q^n \equiv 1 \bmod \ell$,
contradicting 
the order of $q$.
Similar arguments work for type ${}^2A_n$ ($n\geq 2$).
\end{proof}

\section{Useful facts about isogenies}
\begin{lem}
\label{lem:trivial-iso}
Let $\ff$ be a finite field, and
$\zeta:G\twoheadrightarrow G^{\prime}$
a central $\ff$-isogeny of connected reductive $\ff$-groups.
If the kernel has no nontrivial $\ff$-points, then $\zeta$
induces an isomorphism
$G(\ff) \overset{\sim}{\longrightarrow} G'(\ff)$.
\end{lem}

\begin{proof}
It immediately follows from \cite{steinberg:endomorphisms}*{\S4.5}
that the cokernel
of the embedding $G(\ff)\rightarrow G^\prime(\ff)$
is trivial and therefore the map is an isomorphism. 
\end{proof}

\begin{lem}
\label{lem:odd-iso}
Let $F$ be either a finite or a nonarchimedean local field, and
$\zeta:G\longrightarrow G^{\prime}$
a central $F$-isogeny of connected reductive $F$-groups,
with kernel $K$ of odd order.
Then $G(F)$ admits
irreducible self-dual cuspidal representations
if and only if $G^{\prime}(F)$ does. 
\end{lem}

\begin{proof}
Let $\tau^{\prime}$ be a self-dual cuspidal representation of $G^{\prime}(F)$
and let $\tau$ denote its restriction to $G(F)$ along $\zeta$.
If $F$ is finite or has characteristic zero,
then the map from $G(F)$ into $G^{\prime}(F)$
has a finite cokernel of odd order,
so $\tau$ decomposes into
a finite direct sum of an odd number of irreducible representations.
More generally, the cokernel could be a product of a 
finite group of odd order and a pro-$p$-group,
where $p$ is the characteristic of $F$.
From \cite{silberger:isogeny-restriction},
we still have that $\tau$ decomposes into a finite direct sum of irreducible
components, and so the number of components must again be odd.
Therefore, at least one of the components must be self-dual.
Also by \cite{silberger:isogeny-restriction}*{Lemma 1},
since $\tau^{\prime}$ is cuspidal, 
then so is $\tau$.
Therefore, a self-dual component of $\tau$ must also be cuspidal.

Conversely,
if $(\tau,V)$ is an irreducible self-dual cuspidal representation of $G(F)$,
then
the restriction to $K(F)$ of the central character of $\tau$ 
must be self-dual, and so must take values in $\{\pm 1\}$.
But since $|K(F)|$ is odd, $\tau|_{K(F)}$ must be trivial. Therefore
$\tau$ descends to a self-dual representation of $\zeta(G(F))$.

Since $(\tau,V)$ is cuspidal,
we claim that the representation 
$(\Ind_{\zeta(G(F))}^{G^{\prime}(F)}\tau,W)$ 
is also cuspidal.
Recall that a representation $(\pi^\prime,V^{\prime})$ of $G^\prime(F)$ is 
cuspidal if and only if $V_{N^{\prime}(F)}^{\prime}=0$ for all parabolic 
$F$-subgroups $P^\prime$ of $G^\prime$ with Levi decomposition
$P^{\prime}=M^{\prime}N^{\prime}$.
Let $P$, $M$, and $N$ be the inverse images of $P'$, $M'$, and $N'$
under $\zeta$.
Then $P=MN$ is a 
parabolic $F$-subgroup of $G$,
and $\zeta$ induces an isomorphism from $N$ to $N^\prime$.
There is a natural isomorphism
$(W)_{N^{\prime}(F)}\cong\Ind_{\zeta(M(F))}^{M^{\prime}(F)}(V_{N(F)})$.
Since $\tau$ is cuspidal, $V_{N(F)}=0$, and therefore $W_{N^{\prime}(F)}=0$.
This proves that $(\Ind_{\zeta(G(F)}^{G^\prime(F)}\tau,W)$
is cuspidal. 

The representation $\Ind_{\zeta(G(F))}^{G^{\prime}(F)}\tau$
is self-dual.
It decomposes into an odd number of irreducible components
by \cite{silberger:isogeny-restriction} and Frobenius reciprocity
(in the generality presented in
\cite{bushnell-henniart:gl2-book}*{\S2.4}),
and therefore at least one of the components must be self dual. 
\end{proof}

\begin{lem}
\label{lem:odd-iso-DL}
Let $\ff$ be a finite field, and
$\zeta:G\longrightarrow G^{\prime}$
a central $\ff$-isogeny of connected reductive $\ff$-groups,
with kernel $K$ of odd order.
If $G(\ff)$ admits
irreducible self-dual Deligne-Lusztig cuspidal representations,
then so does $G^{\prime}(\ff)$.
\end{lem}

\begin{proof}
Suppose $T$ is a maximal elliptic $\ff$-torus in $G$,
and $\chi$ is a character of $T(\ff)$ that is in general
position and conjugate self-dual.
Then the restriction of $\chi$ to $K$ must be self-dual, and 
thus trivial.
Let $T'$ be the image of $T$ under $\zeta$.
Then $\chi$ factors through $\zeta$ to give a character $\chi'$
of the image $A$ of $T(\ff)$ in $T'(\ff)$,
and $\chi'$ is also in general position and conjugate self-dual.
Every extension of $\chi'$ from $A$ to $T'(\ff)$ is in general position.
Since the index of $A$ in $T'(\ff)$ is odd, at least one of these
extensions must be conjugate self-dual.
\end{proof}

\section{Simply connected groups}

We collect here some of the assumptions that we will sometimes
have to make about a connected reductive $\ff$-group $G$.

\begin{hyp}
\label{hyp:SU-restrictions}
The group $G$ has no factor of type $^2A_k(q)$, where
\begin{enumerate}[(a)]
\item
$k=2$ and $q=2$;
\item
$k=2$ and $q\in{\{3,4\}}$; or $k=3$ and $q\in{\{2,3,5\}}$;
or $k=4$ and $q\in{\{2,3, 4,5\}}$.
\end{enumerate}
\end{hyp}

We will need to assume (a) to assure that $G$ has irreducible
Deligne-Lusztig cuspidal representations
(but see Remark \ref{rem:SU-small} concerning $\PU(3)$).
Parts (a) and (b) together will assure that $G$ 
has irreducible, \emph{self-dual}, Deligne-Lusztig supercuspidal
representations.

\begin{thm}
\label{thm:sc}
Suppose that $G$ is simply connected.
Then $G(\ff)$ admits irreducible cuspidal representations.
If $G$ satisfies Hypothesis \ref{hyp:SU-restrictions}(a),
then $G(\ff)$ admits
irreducible \emph{Deligne-Lusztig} cuspidal representations.
Moreover, if $G$ also satisfies Hypothesis \ref{hyp:SU-restrictions}(b),
then 
the following are equivalent:
\begin{enumerate}
\item
$G$ has no factor of type $A_n$ ($n$ even);
\item
$G(\ff)$ admits irreducible, \emph{self-dual}
cuspidal representations.
\item
$G(\ff)$ admits irreducible, self-dual,
\emph{Deligne-Lusztig} cuspidal representations.
\end{enumerate}
\end{thm}

\begin{proof}
Since $G$ is a direct product of absolutely almost simple groups,
we may assume that $G$ is absolutely almost simple.

If $G$ has type $^2A_2(2)$, then our result will follow
from Remark \ref{rem:SU-small},
so assume from now on that $G$ has a different type.

If $G$ has type $^2A_n$ ($n \geq 2$), $^2D_n$ ($n \geq 4$, odd),
or $^2E_6$, then our result will follow
from Proposition \ref{prop:SU},
\ref{prop:2Dn}, or
\ref{prop:2E6}, respectively,
so assume from now on that $G$ has a different type.

When $G$ has type other than $A_n$ ($n$ even),
then $h^\prime:=h/t$ is even \cite{reeder:torsion}*{Section 5, table}.
Therefore our result follows from Proposition \ref{prop:coxeter}
unless $G$ has type $G_2(2)$.

Suppose $G$ has type $G_2(2)$.
Let $T$ be the Coxeter torus of $\SL(3)$, which is a subgroup
of $G_2$.
Then $T$ is an anisotropic maximal $\ff$-torus in $G$.
If $w$ is the Coxeter element of $W(G_2)$, then $T$ corresponds
to $w^2$, which is the Coxeter element of the subgroup $W(A_3)$.
Since $w^2$ is Coxeter for $A_3$ we know the structure of $T(\ff)$,
and thus of its character group $L$:
They are cyclic of order $(2^3-1)/(2-1) = 7$.
The centralizer of $w^2$ in $W(G_2)$ is $\langle w \rangle$,
a group of order $6$.
Since the automorphism of $L$ is also cyclic of order $6$,
it must be the case that the action of $w$ on $L$ generates
all of the automorphisms of it.
Every element of $L$ is thus conjugate self-dual,
and every nontrivial element of $L$ is in general position.

We have shown that under Hypothesis \ref{hyp:SU-restrictions}(a,b),
statement (1) above implies statement (3).
Suppose $G = \SL_{n+1}$, with $n$ even.
From 
\cite{carter:conjugacy-weyl}*{Prop.\ 23},
the elliptic elements of the Weyl group are precisely the Coxeter
elements.
For $\GL_{n+1}(\ff)$ and thus for $\PGL_{n+1}$, 
all cuspidal representations are of Deligne-Lusztig type.
Therefore, by Proposition \ref{prop:coxeter},
$\PGL_{n+1}(\ff)$ has no irreducible self-dual 
cuspidal representations.
By Lemma \ref{lem:odd-iso}, the same is true for $G(\ff)$. 
\end{proof}

\begin{prop}
\label{prop:2E6}
Suppose $G$ is absolutely almost simple of type ${}^2E_6$.
Then $G(\ff)$ admits self-dual Deligne-Lusztig cuspidal representations. 
\end{prop}

\begin{proof}
With notation as in Section \ref{sec:DL-reps},
write $\sigma_{\ff}=\sigma_{0}\cdot q$
where $\sigma_{0}$ is an involution of $X$.
From \cite{springer:regular-elements}*{Table 8, \S 6.12}, we see that
${}^2E_6$ admits a regular element $w:=\omega\sigma_{0}$ of order
$h=12$ in the twisted Weyl group $W\sigma_{0}$.
Also by \emph{loc.~cit.}, its characteristic polynomial is $\Phi_{12}\Phi_{6}$,
so none of the eigenvalues of $\omega\sigma_{0}$ are $1$,
i.e., $\omega\sigma_{0}$ is elliptic.
Let $T=T_{\omega}$ (resp.\ $L=L_{\omega}$)
be the associated elliptic torus
(resp.\ group of complex characters of $T(\ff)$)
as in the notations of Section \ref{sec:DL-reps}.
The centralizer $\Omega$ of $w$
in $W$ has exactly one \emph{degree} (see \cite{springer:regular-elements}*{\S 2.3} for the definition of degree), which is $12$.
Therefore by \cite{springer:regular-elements}*{Corollary 3.3},
$\Omega$ is a cyclic subgroup of $W$ of order $12$.
By \cite{birkhoff-vandiver}*{Theorem V},
there exists a prime $\ell$ such that the multiplicative
order of $q$ mod $\ell$ is $12$. The cyclic group generated by
$w^{2}$ in $W$ is a subgroup of $\Omega$ of index $2$, and $w$
acts on $L$ by multiplication by $q$. As in the proof of Theorem
\ref{thm:sc}, there exists a cyclic subgroup $C$ of $L$ of order $\ell$ such
that the generator $v$ of $C$ is in general position with respect
to the subgroup generated by $w$ in $W$. Let $\tau$ be a generator
of $\Omega$ such that $\tau^{2}$ acts by $q^{2}$ on $C$.  Write
$v^{\prime}=\tau v$. If $v^{\prime}\notin C$, then the subgroup
$C^{\prime}$ of $L$ generated by $v^{\prime}$ is cyclic of order
$\ell$. In this case, it is clear that $v$ is in general position
with respect to $\Omega$. Now if $\tau$ stabilizes $C$, then $\tau$
acts by multiplication by an integer $r$ and $r^{2}\equiv q^{2}$
mod $\ell$, i.e., $r\equiv\pm q$ mod $\ell$. In either case, $r$
has order $12$ mod $\ell$. Therefore $v$ is in general position
with respect to $\Omega$.
Since $r^6$ acts by $-1$ on $v$, $v$ is also conjugate self-dual.
\end{proof}

\section{Groups of type $^2D_n$ ($n \geq 4$ odd) and $^2A_n$}
\label{sec:classical}

Here we conclude some leftover business from the proof of Theorem \ref{thm:sc}:
certain groups of classical type that require special handling.
We include some statements about non-simply-connected groups here
because we find it convenient to do so.

\begin{prop}
\label{prop:2Dn}
Suppose $n\geq 4$ is odd,
and $G_0$ is an absolutely almost simple group of type $^2 D_n$.
Then $G_0(\ff)$
has irreducible, self-dual, Deligne-Lusztig cuspidal representations.
Moreover, some of these representations come from
a character of odd order.
\end{prop}

\begin{prop}
\label{prop:SU}
Suppose $G_0$ is an absolutely almost simple group of type ${}^2A_n$.
Then $G_0(\ff)$ admits irreducible, cuspidal representations.
If $G_0$ satisfies Hypothesis \ref{hyp:SU-restrictions}(a),
then $G_0(\ff)$ also admits irreducible,
\emph{Deligne-Lusztig} cuspidal representations. 
If $G_0$ also satisfies Hypothesis \ref{hyp:SU-restrictions}(b),
then $G_{0}(\ff)$ admits irreducible,
\emph{self-dual}, Deligne-Lusztig cuspidal representations.
\end{prop}

Our proofs require some notation and background.
Let $E/\ff$ be the quadratic extension.  For each natural number $d$,
let $E_d$ and $\ff_d$ denote the extensions of $E$ and $\ff$ of degree $d$.
Let $\TT_d$ denote the kernel of the
norm map from $E_d\mult$ to $\ff_d\mult$.

Let $G = \U(m)$ or the nonsplit form of $\SO(m)$
(with $m$ even in the latter case).
If $T$ is a maximal elliptic torus in $G$,
then $T(\ff)$ is isomorphic to a direct product
$\prod_{i=1}^r \TT_{d_i}$,
where $\sum d_i = m$.
If $G$ is unitary, then we require all $d_i$ to be odd.
If $G$ is orthogonal, then the number of factors $r$ must be odd.
Conversely, given any product as above, there is at least one
associated maximal elliptic torus $T\subset G$.

The action on $T(\ff)$
of the rational Weyl group $W$ of $T$ in $G$ is generated by
the action of $\Gal(E_{d_i}/E)$ on each factor $\TT_i$,
together with those permutations in $S_r$ that give
rise to automorphisms of $T(\ff)$ via permuting factors in the product
above.

Thus, each elliptic torus in $\U(m)$ or $\SO(m)$
is a product of Coxeter tori from smaller-rank groups.
Specifically, $\TT_d$ is isomorphic to the group of rational 
points of the Coxeter torus $T'$ in $\Sp(2d)$.
The Weyl group of $T'$ in $\Sp(2d)$ acts on $T'(\ff)$ via
$\Gal(E_d/\ff)$.
Restricting this action to $\Gal(E_d/E)$, we obtain
our action on the factor $\TT_d$ of $T(\ff)$,
where a generator acts via multiplication by $q^2$.

\begin{lem}
\label{lem:v-element}
Let $T$ be an elliptic torus such that $T(\ff)$ is the direct product of two
copies of $\TT_{k_1}$, zero or two copies of $\TT_{k_2}$
(with $k_1$ and $k_2 > 1$, $k_1\neq k_2$)),
and $r$ copies of $\TT_1$,
where $0\leq r \leq 3$.
Then the character group $L$ of $T(\ff)$
has an element that is conjugate
self-dual and in general position.
If $r \leq 1$, then $L$ has such an element
whose Weyl orbit lies in a subgroup
of order 
coprime to $2$ and, if $q\neq2$, also coprime to $q+1$.
\end{lem}

\begin{defn}
\label{def:good}
Let $T$ and $r$ be as in Lemma \ref{lem:v-element}.
We call $T$ \emph{good} 
if $r\leq1$ and \textit{bad} otherwise.
If $T$ is bad, then we call the product of factors 
other than $\TT_1$ the \emph{good part}
and the product of the rest the \emph{$\TT_1$ part}.  
If $G_{0}$ is simply connected of type ${}^2A_n$ (resp.\ ${}^2D_n$)
and $T_{0}$ is a torus in $G_{0}$, then we say that $T_{0}$ is
\emph{fine} if it comes from (resp.\ is a pull back of)
a good torus of $\U(n)$ (resp.\ $\SO(2n)$).
\end{defn}

\begin{proof}[Proof of Lemma \ref{lem:v-element}]
Write $L$ as a product $(L_{k_1} \times L_{k_1}) \times \cdots$
analogously to our product decomposition for $T(\ff)$.
Since $L$ is thus a direct product of subgroups that are
preserved by the action of the Weyl group, we may consider
each of these subgroups independently.

Suppose $L = L_k \times L_k$ with $k>1$. Assume also that $k\neq3$ if $q=2$.
From Proposition \ref{prop:coxeter} and Lemma \ref{lem:coprime},
we can choose an element $v_k \in L_k$
that is in general position (with respect to the action
of $\Gal(E_k/\ff)$)
and of odd order $\ell$.
Since the order of $q\bmod \ell$ is $2k>2$
(recall that $\TT_k$ is the $\ff$-points of a 
Coxeter torus of $\Sp(2k)$),
we have that $\ell$ is also coprime to $q+1$.
If $k$ is even, then we can and do choose $v_k$ to be conjugate self-dual,
and let $v'_k = qv_k$.
If $k$ is odd, then let $v'_k=-v_k$.
In either case, note that $v'_k$ is not in the orbit of $v_k$
under the action of $\Gal(E_k/E)$,
so
$v:= (v_k, v'_k)$ is conjugate self-dual and in general position.
The Weyl orbit lies inside $\langle v_k \rangle \times \langle v'_k\rangle$,
a group of order $\ell^2$.

When $(k,q)=(3,2)$, then $L_{k}$ is a cyclic group of order $9$.
In this case, choose $v_{k}$ to be a generator of $L_{k}$
and let $v_{k}^{\prime}=-v_{k}$.
Then $v=(v_{k},v_{k}^{\prime})$ is conjugate self-dual in general
position. 

If $L = L_1$, then the Weyl group is trivial,
and $v=0$ is conjugate self-dual and in general position.

If $L$ is a product of two or three copies of $L_1$, then let $v_1$
be a generator of $L_1$, and let $v:= (v_1, -v_1)$ or $(v_1, -v_1,0)$
according as $r=2$ or $3$.
Then $v$ is conjugate self-dual and in general position.
\end{proof}

\begin{proof}[Proof of Proposition \ref{prop:2Dn}]
Write $n=2k+1$ with $k>1$.
Choose a maximal elliptic torus $T$ such that
$T(\ff) \cong \TT_k\times \TT_k \times \TT_1$.
Then the group $L$ of complex characters of $T(\ff)$
has the form $L_k \times L_k \times L_1$,
where $L_i$ is the character group of $\TT_i$.
From Lemma \ref{lem:v-element},
there is an element $v\in L$ that is conjugate
self-dual and in general position,
and its Weyl orbit lies inside a subgroup $C$ of $L$ of odd order.

The existence of $v$ proves our result for $\SO(2n)$.

Let $T'$ denote the inverse image of $T$ in $\Spin(2n)$.
Then the character group $L$ of $T(\ff)$ surjects
onto the character group $L'$ of $T'(\ff)$, with kernel
of order $2$.
This surjection is equivariant with respect to the action of the Weyl group.
Since $C$ has odd order, it is isomorphic to its image
under this surjection.
Therefore, the image of $v$ in $L'$ is conjugate self-dual and general
position,
proving our result for $\Spin(2n)$.

Similar reasoning proves our result for the adjoint group
of type $^2D_n$.
\end{proof}

One can obtain crude results for unitary groups $\U(n)$ by choosing
our elliptic torus in a way that is independent of $n$.
We include such results here, since they are the best possible when
$n$ is small.

\begin{lem}
\label{lem:U-crude}
If $q$ is even and $n$ is odd,
then assume $q \geq n - 1$.
Otherwise, assume $q\geq n$.
Then the following are true:
\begin{enumerate}[(a)]
\item
\label{item:Un}
$\U(n)(\ff)$ 
admits irreducible, self-dual, Deligne-Lusztig cuspidal representations
that descend to $\PU(n)(\ff)$.
\item
\label{item:SUn}
If $q$ is even or $n$ is odd, then suppose that $q \geq n$.
If $q$ is odd and $n$ is even, then suppose that $q > n+1$.
Then some of our representations of $\U(n)(\ff)$ above remain irreducible
upon restriction to $\SU(n)(\ff)$.
\end{enumerate}
\end{lem}

The proof for $\U(n)$
was suggested to us by Dipendra Prasad.

\begin{proof}
We have a maximal elliptic torus $T$ in $\U(n)$ such that
$T(\ff) \cong \prod_{i=1}^n \TT_1$.
Let $L$ be the group of complex characters of $T(\ff)$.
Then $L$ is a direct product of $n$ copies of a cyclic group $C$
of order $q+1$.
Write $n= 2k$ or $n=2k+1$ according as $n$ is even or odd.
Let $c$ be a generator of $C$,
and let $v = (c,-c, 2c, -2c, \ldots , kc, -kc)$ if $n=2k$,
or $(0, c, -c, \ldots , kc, -kc)$ if $n=2k+1$.
Our assumption on $q$ assures that the coordinates of $v$
are all distinct.
Thus, $v$ is in general position,
and it is easily seen to be conjugate self-dual,
thus providing an irreducible, self-dual, Deligne-Lusztig
cuspidal representation of $\U(n)(\ff)$.
Since the coordinates of $v$ sum to $0$, this representation
has trivial central character, and so gives us a representation
of $\PU(n)(\ff)$ as well,
proving part (\ref{item:Un}).

Now consider the torus $T' := T\cap \SU(n)$ in $\SU(n)$.
It will be enough to show that the image of $v$ in
the group $L'$ of characters of $T'(\ff)$
is still in general position.
Note that $L'$ is the quotient of $L$
by the diagonally embedded subgroup $\diag(C)$.
If $q=n-1$, then it is easy to see that our element $v\in L$ above
is, up to permutations, the only element in general position, and that
its image in $L'$ is not in general position.
Therefore, we must and do assume from now on that $q \geq n$.

Thus, the set $\mathcal{C}$ of elements of $C$
that appear as coordinates in $v$ is a proper subset of $C$.
Since $v$ is in general position,
so is its image in $L'$,
provided that the set $\mathcal{C}$
is not invariant under addition by any nonzero element of $C$.
If $n$ is odd or $q$ is even, then indeed $\mathcal{C}$ is not invariant.
Suppose $n$ is even and $q$ is odd, and $\mathcal{C} + \lambda = \mathcal{C}$
for some nonzero $\lambda \in C$.
Then $q = n + 1$,
and $2\lambda = 0$.
Therefore, if $q \geq n+1$
then $\mathcal{C}$ is not invariant,
proving part (\ref{item:SUn}).
\end{proof}

\begin{rem}
\label{rem:SU-small}
We gather together some facts about the unitary groups that
Hypothesis \ref{hyp:SU-restrictions} excludes.
\begin{enumerate}[(a)]
\item
From Lemmas \ref{lem:U-crude}(a), \ref{lem:odd-iso}, and \ref{lem:trivial-iso},
if $G$ is an isogenous image of $\SU(n)$ ($n=3$, $4$, or $5$),
then $G(\ff)$ has irreducible, self-dual cuspidal representations
except possibly in the following cases:
$G = \PU(4)$ and $q \in \{2,3\}$;
$G$ is an isogenous pre-image of $\PU(4)$ and $q \in \{2,3,5\}$;
$n=5$ and $q \in \{2,3\}$.
%
%
%
\item
Let $G=\U(n)$, $\SU(n)$, or $\PU(n)$, where $n=3$, $4$, or $5$.
Then the only elliptic tori in $G$ are the Coxeter torus
and the torus used in the proof of Lemma \ref{lem:U-crude}.
The Coxeter torus has no conjugate self-dual characters.
Therefore, $G(\ff)$ has no
irreducible self-dual Deligne-Lusztig cuspidal representations
unless some were constructed in Lemma \ref{lem:U-crude}.
\item
Independent of $q$, $\SU(3)$ has one cuspidal unipotent
representation which, by uniqueness, is self-dual.
\end{enumerate}
\end{rem}

\begin{proof}[Proof of Proposition \ref{prop:SU}]
Our claims on the existence of self-dual cuspidal representations,
and of self-dual \emph{Deligne-Lusztig} cuspidal representations,
for $\SU(4)(\ff)$ and all isogenous images of $\SU(3)(\ff)$
and $\SU(5)(\ff)$
follow from Remark \ref{rem:SU-small}
and Lemma \ref{lem:U-crude}.

Suppose from now on that $n>5$.
Let $G=\U(n)$ and $G^\prime=\SU(n)$.

Suppose $n\equiv 2$ or $3 \bmod{4}$.
Write $n=2k$ or $2k+1$, where $k>1$ is odd.
Choose an elliptic torus $T \subset G$ such that
$T(\ff) \cong \TT_k \times \TT_k$ or
$\TT_k \times \TT_k \times \TT_1$
according as $n$ is even or odd.
From Lemma \ref{lem:v-element},
we can choose an element $v$ in the character group $L$ of $T(\ff)$
that is conjugate self-dual and in general position. If $(k,q)\neq(3,2)$,
the Weyl orbit of $v$ generates a group $C$ of order coprime
to $q+1$.

Now let $T' = T\cap \SU(n)$.
The restriction map induces a surjection
from the character group $L$ of $T(\ff)$ onto the character group
$L'$ of $T'(\ff)$ with kernel of order  $q+1$.
Therefore,
if $(k,q)\neq(3,2)$,
the group $C$ is isomorphic to its image $C'$ under this surjection.
The surjection is equivariant with respect to the action of the Weyl group.
Therefore, the image $v'$ of $v$ in $L'$ is conjugate self-dual
and in general position. If $(k,q)=(3,2)$, then again it is easy to check that
the  image $v'$ of $v$ in $L'$ remains conjugate self-dual
and in general position.
Suppose $G''$ is an isogenous image of $\SU(n)$, and $T''$ is the image
of $T'$ under the isogeny.
Letting $L''$ be the character group of $T''(\ff)$, we obtain
a map $L''\longrightarrow L'$ whose kernel and cokernel
have order dividing $q+1$.
Therefore, if $(k,q)\neq(3,2)$, $C'$ lies in the image of $L''$,
and its preimage contains
a subgroup $C''$ isomorphic to $C'$.  The preimage of $v''$ of $v'$
is then conjugate self dual and in general position.
If $(k,q)=(3,2)$, then the kernel and cokernel of $L''\rightarrow L'$
have odd order. Therefore the existence of conjugate self-dual element
$v''\in L''$ follows from Lemma \ref{lem:odd-iso-DL}.

Now consider $n=4k$.
If $k>1$ is odd, then choose an elliptic torus $T \subset G$ 
such that
$T(\ff)=\TT_{k+2} \times \TT_{k+2} \times \TT_{k-2} \times \TT_{k-2}$.
If $k$ is even, then 
choose $\TT_{k+1}\times \TT_{k+1} \times \TT_{k-1} \times \TT_{k-1}$.
Finally for the case $n=4k+1$, choose
$\TT_{2k-1}\times \TT_{2k-1}\times \TT_1\times \TT_1 \times \TT_1$.
In all these cases,
Lemma \ref{lem:v-element} shows the existence of a conjugate
self-dual element of $L$ in general position.
Note that the tori constructed above are good
except when $n \equiv 1 \bmod 4$, or 
$n=8$ or $12$.
In the cases where the torus is good,
we have such an element $v \in L$ whose Weyl
orbit lies in a group of order coprime to $q+1$.
Such a group must have an isomorphic image in the character group
$L'$ of $T'(\ff)$, where $T' = T'\cap \SU(n)$,
and an isomorphic preimage in the character group of the corresponding
torus in any isogenous image of $\SU(n)$.

The result about isogenous images of $\SU(n)$
for $n=4k+1$ follows from Lemma \ref{lem:odd-iso-DL}.

We will deal with the cases of $\SU(8)$ and $\SU(12)$
in Lemma \ref{lem:SU8,12},
with the isogenous images of $\SU(4)$, $\SU(8)$, and $\SU(12)$
in Proposition \ref{prop:semisimple}.
\end{proof}

\begin{lem}
\label{lem:SU8,12}
Let $n=8$ or $12$.
Then $\U(n)(\ff)$ has a self-dual, Deligne-Lusztig cuspidal representation
that has trivial central character, and whose restriction
to $\SU(n)(\ff)$ remains irreducible.
\end{lem}

\begin{proof}
Choose an elliptic torus $T^{\mathrm{u}}\subset \U(n)$ such that
$T^{\mathrm{u}}(\ff) =  \TT_k \times \TT_k \times \TT_1 \times \TT_1$,
where $k=(n/2)-1$.
Let $T = T^{\mathrm{u}} \cap \SU(n)$,
and $T_\der$ the image of $T^{\mathrm{u}}$ in $\PU(n)$.
Let $L^\mathrm{u}$, $L$, and $L_\der$ denote the groups of characters of
the groups of rational points of these tori.
Then $L^\mathrm{u}$ is a product
$L_k \times L_k \times L_1 \times L_1$,
where each $L_i$ is cyclic of order $q^i+1$.
Write 
$v^{\mathrm{u}}= (c,-c,d,-d)$,
where $c\in L_k$ generates a subgroup $C_k$ of prime order $\ell$
which is coprime to $q_{i}+1$, and $d$ is a generator of $L_1$. 
The element $v^{\mathrm{u}}\in L^{\mathrm{u}}$
is conjugate self-dual and in general position.
Regarding $L_1$ as a subgroup of $L_k$, we see that the sum
of the coordinates of $v^{\mathrm{u}}$ is $0$, meaning
that $v^{\mathrm{u}} \in \im( L_\der \longrightarrow L^{\mathrm{u}})$.
Therefore, our Deligne-Lusztig cuspidal representation
of $\U(n)(\ff)$ constructed from $v^{\mathrm{u}}$ is an isogeny
restriction of a representation of $\PU(n)(\ff)$.

Since $L_1$ embeds in each factor of $L^{\mathrm{u}}$
we have a diagonally embedded subgroup $\diag(L_1) \subset L^{\mathrm{u}}$,
and $L$ is the quotient $L^{\mathrm{u}} / \diag(L_1)$.

Let $v$ denote the image of $v^{\mathrm{u}}$ in $L$.
Then $v$ is obviously conjugate self-dual.
It remains to see that $v$ is in general position.
That is, we need to see that for nonzero $\lambda \in L_1$,
$v^{\mathrm{u}} + (\lambda,\lambda,\lambda,\lambda)$
cannot be a Weyl conjugate of 
$v^{\mathrm{u}}$.
But this follows from the fact that the Weyl orbit of $v^{\mathrm{u}}$
is contained in $C_k \times C_k \times L_1 \times L_1$,
and $c+\lambda\notin C_k$ since
$\gcd(\ell,q+1)=1$.
\end{proof}

\section{Semisimple groups}

We assume now that $G$ is semisimple.
Consider the central $k$-isogeny $\widetilde{G}\longrightarrow G$,
where
$\widetilde{G}$ is the simply
connected cover of $G$, and let $\widetilde{T}_{0}$ be the maximal
torus of $\widetilde{G}$ that surjects to $T_{0}$ under this isogeny.
Write $\widetilde{G}=\prod_{i\in I}\Rest_{E_{i}/\ff}\widetilde{G}_i$
(resp.\ $\widetilde{T}_{0}=\prod_{i\in I}\Rest_{E_{i}/\ff}\widetilde{T}_{0i}$)
where $I$ is a finite indexing set and the groups $\widetilde{G}_i$
(resp.\ $\widetilde{T}_{0i}$)
are absolutely almost simple (resp.\ maximally split maximal  
tori of $\widetilde{G}_i$),
and $E_{i}/\ff$ are finite extensions of degree $n_i$.
Let $\widetilde{X}$
(resp.\ $\widetilde{X}_{i}$)
denote the character lattice of $\widetilde{T}_{0}$
(resp.\ $\widetilde{T}_{0i}$).
Let $\Gamma_{\ff}$ (resp.\ $\Gamma_{E_{i}}$) denote the absolute Galois
group of $\ff$ (resp.\ $E_{i}$).
They are cyclic groups generated by $\sigma_{\ff}$
(resp.\ $\sigma_{\ff}^{n_{i}}$).
The isogeny induces an inclusion of lattices
\[
X\hookrightarrow\widetilde{X}
=
\bigoplus_{i\in I}\Ind_{\Gamma_{E_{i}}}^{\Gamma_{\ff}}\widetilde{X}_{i} \, .
\]
The Weyl group $W$ is a product of Weyl
groups $\prod_{i\in I}W_{i}^{n_{i}}$ where $W_{i}$
is the Weyl group of $\widetilde{G}_i$.
Let $\omega_{i}$ be a ${\sigma_\ff}^{n_i}$-elliptic element of $W_{i}$
and let $\overset{\circ}{\omega}_{i}$ be the element of 
$W_{i}^{n_{i}}$ that acts on
$\Ind_{\Gamma_{E_{i}}}^{\Gamma_{\ff}}\widetilde{X}_{i}$
via the action of $\omega_{i}$ on $\widetilde{X}_{i}$.
Write $\omega=\prod \overset{\circ}{\omega_{i}}$,
$\widetilde{T}=(\widetilde{T}_{0})_{\omega}$,
and $\widetilde{T}_i=(\widetilde{T}_{0i})_{{\omega}_i}$. 
Let 
\begin{align*}
\widetilde{L}
& :=  \Hom(\widetilde{T}(\ff),\mathbb{C}^{\times}),\\
L
& :=  \Hom(T(\ff),\mathbb{C}^{\times}),\\
\widetilde{L}_{i}
& :=
\Hom(
	\Rest_{E_{i}/\ff}\widetilde{T}_{i}(\ff)
	,
	\mathbb{C}^{\times}
).
\end{align*}
Let
$z_{i}:=
	\ker(\Rest_{E_{i}/\ff}\widetilde{T}_{i}\rightarrow T)$.
Then
$\coker(L\rightarrow\widetilde{L}_{i})
=\Hom(z_{i}(\ff),\mathbb{C}^{\times})$.
Consequently, the order of $\coker(L\rightarrow\widetilde{L}_{i})$
divides the order of $\widetilde{Z}_{i}(E_{i})$, where $\widetilde{Z}_{i}$
denotes the center of $\widetilde{G}_{i}$.

We let $h_{i}$ denote the Coxeter number of $\widetilde{G}_{i}$
and $q_i:=q^{n_i}$.

\begin{prop}
\label{prop:semisimple}
Suppose $G$ is semisimple
and satisfies Hypothesis \ref{hyp:SU-restrictions}(a).
Then $G(\ff)$ admits irreducible, cuspidal, Deligne-Lusztig representations.
Moreover, if $G$ also satisfies Hypothesis \ref{hyp:SU-restrictions}(b)
and has no factor of type
$A_n$ ($n$ even),
then $G(\ff)$ admits irreducible, \emph{self-dual},
cuspidal, Deligne-Lusztig representations.
\end{prop}

\begin{proof}
Recall that we write the simply connected cover $\widetilde{G}$
of $G$
as $\prod_{i\in I} \Rest_{E_i/\ff} \widetilde{G}_i$.

We first reduce to the case where none of the factors
in this product
has any of the types that required special handling
in the proofs of Proposition \ref{prop:coxeter} and Theorem \ref{thm:sc}.
We have a central $\ff$-isogeny $H_1\times H_2 \longrightarrow G$,
where $H_1$ is a direct product of groups having one of the types
$A_5(2)$, $C_3(2)$, $D_4(2)$, $G_2(2)$, or ${}^2A_3(2)$,
no factor of $H_2$
has any of those types,
and the restriction of the isogeny to $H_2$ has trivial kernel.
Since the center of $H_1$ has no nontrivial rational points,
Lemma \ref{lem:trivial-iso} shows that $H_1(\ff) \times H_2(\ff)$
is isomorphic to $G(\ff)$.
From Theorem \ref{thm:sc}, 
every factor of $H_1(\ff)$ has irreducible Deligne-Lusztig
cuspidal representations and if $H_1$ has no factor of type ${}^2A_3(2)$,
then it also 
has self-dual Deligne-Lusztig
cuspidal representations.
Therefore, we may replace $G$ by $H_2$,
and thus assume from now on that
$G$ has no factor of type
$A_5(2)$, $C_3(2)$, $D_4(2)$, $G_2(2)$, or ${}^2A_3(2)$.

Write the indexing set $I$ as a disjoint union
$I=I_1\sqcup I_2 \sqcup I_3$,
where
\[
\text{$\widetilde{G}_{i}$ is of type}
\begin{cases}
A_1 & \text{if $i\in I_1$},\\
\text{${}^2A_k$ ($k>1$ odd) or ${}^2D_k$ ($k>4$ odd)} & \text{if $i\in I_2$},\\
\text{something else} & \text{if $i\in I_3$}.
\end{cases}
\]

For $i\in I$, let $\widetilde{T}_i^\mathrm{c}$
(resp.\ $\widetilde{T}^\mathrm{c}$, $T^\mathrm{c}$)
denote the Coxeter torus of $\widetilde{G}_i$
(resp.\ $\widetilde{G}$, $G$).
Let
$\widetilde{L}_{i}^\mathrm{c}
:=
\Hom(\Rest_{E_{i}/\ff}\widetilde{T}_{i}^\mathrm{c}(\ff),\mathbb{C}^{\times})$
and $L^\mathrm{c} := \Hom(T^\mathrm{c}(\ff),\mathbb{C}^{\times})$.

As in Proposition \ref{prop:coxeter},
for $i\in I\backslash I_1$,
let $\widetilde{u}_i$ be an element in general position in 
$\widetilde{L}_{i}^\mathrm{c}$
of order $\ell_i$, where $\ell_i$ is a prime such that the
multiplicative order of $q_i$ mod $\ell_{i}$ is  $h_{i}$.
The Weyl orbit of $\widetilde{u}_i$ lies in a cyclic group $\widetilde{C}_i$
of order $\ell_i$. Lemma \ref{lem:coprime}
implies that 
$\ell_{i}$ is coprime to
$|\coker(L^\mathrm{c}\rightarrow\widetilde{L}_{i}^\mathrm{c})|$.
Therefore $\widetilde{u}_i$ lifts to an element $u_i\in L^\mathrm{c}$ such that
the orbit of $u_i$ is contained in
a subgroup $C_i^\mathrm{c}\cong \widetilde{C}_i$ of $L^\mathrm{c}$. 

For $i\in I_1$, 
define $u_i$ to be the image in
$L^\mathrm{c}$ of a generator of the cyclic group $\widetilde{X}_i$. 

Write $u=\sum_{i\in I}u_{i}$. 
Then $u$ is in general position.
We had to assume Hypothesis \ref{hyp:SU-restrictions}(a)
in order to construct $u$, so under these conditions,
$G(\ff)$ has an irreducible Deligne-Lusztig cuspidal representation,
as claimed.

Now assume that $G$ also satisfies
Hypothesis \ref{hyp:SU-restrictions}(b),
and has no factor of type $A_n$ for $n$ even.
It only remains to prove that $G(\ff)$ has
irreducible, self-dual Deligne-Lusztig representations.
We first make some more reductions. 

Our central $\ff$-isogeny
$\widetilde{G}\longrightarrow G$
factors into central $\ff$-isogenies
$\widetilde{G}\longrightarrow G'$
and
$G'\longrightarrow G$,
whose the kernels are (respectively)
a $2$-group and a group of odd order.
From Lemma \ref{lem:odd-iso-DL},
it would be enough to show that $G'(\ff)$ has
self-dual, Deligne-Lusztig cuspidal representations.
Therefore, we may replace $G$ by $G'$,
and assume from now on that the kernel of our isogeny
$\widetilde{G} \longrightarrow G$ is a $2$-group.

Therefore, we may write $G = G_0 \times H$,
where $H$ is a direct product of simply connected groups
of type $E_6$, $^2E_6$, or $^2A_n$ ($n>2$ even),
and $G_0$ has no factors of those types.
From Propositions \ref{prop:2E6}
and \ref{prop:SU},
$H(\ff)$ has irreducible, self-dual, Deligne-Lusztig cuspidal representations.
Therefore, we may replace $G$ by $G_0$, and assume from now on that
$G$ has no factors
of type $E_6$, $^2E_6$, or $^2A_n$ ($n>2$ even).

If $q$ is even, then our result follows from Theorem \ref{thm:sc}
and Lemma \ref{lem:trivial-iso}.
Therefore, we may and will assume from now on that $q$ is odd.

For $i\in I\backslash I_2$,
choose $\widetilde{T}_i$ to be $\widetilde{T}_i^\mathrm{c}$.
For $i\in I_2$, choose $\widetilde{T}_i$ to be 
an elliptic torus as in the proofs of
Propositions \ref{prop:2Dn} and \ref{prop:SU}. 
If $\widetilde{G}_i$ admits
fine tori (Definition \ref{def:good}),
we require $\widetilde{T}_i$ to be fine.
These choices of $\widetilde{T}_i$ determine an elliptic torus $T$ in $G$.

For $i\in I_{1}$, 
define $v_{i}$ to be the image in $L$ 
of a generator of the cyclic group $\widetilde{X}_i$. 

For $i\in I_3$,
$\widetilde{u}_i$ lifts to an element $v_i\in L$ such that
the orbit of $v_i$ is contained in
a subgroup $C_i\cong \widetilde{C}_i$ of $L$. 

For $i\in I_{2}$,
let $\widetilde{v}_{i}$ be
a conjugate self-dual element as in Lemma \ref{lem:v-element}.
If $\widetilde{T}_i$ is fine,
Lemma \ref{lem:v-element} implies that the Weyl orbit 
of $\widetilde{v}_{i}$ lies in a subgroup
$\prod_{j\in\mathfrak{S}}{\widetilde{C}_{ij}}$ of $\widetilde{L}_i$
where each $\widetilde{C}_{ij}$
is either trivial
or cyclic of prime order $\ell_{ij}$ which is coprime to
$|\coker(L\rightarrow\widetilde{L}_{i})|$.
Therefore $\widetilde{v}_i$ lifts to an element $v_i$ in $L$
such that the orbit of $v_i$ lies in
a subgroup $\prod_{i\in\mathfrak{S}}C_{ij}$ of $L$
with $C_{ij}\cong \widetilde{C}_{ij}$.

If $\widetilde{G}_{i}$ admits no fine tori, then $\widetilde{G}_{i}=\SU(n)$
with $n=4$, $8$ or $12$.
Let 
$\widetilde{T}_{i}^{\mathrm{u}} \subset \U(n)$
be the elliptic torus containing
our elliptic torus $\widetilde{T}_i \subset \SU(n)$.
Let $T_{i,\der}$ denote the image of $\widetilde{T}_i$ in $\PU(n)$,
and let $L_{i,\der}$ denote the group of characters of $T_{i,\der}(\ff)$.
In the proofs of Lemmas \ref{lem:U-crude} and \ref{lem:SU8,12},
we constructed an element
$\widetilde v_i^{\mathrm{u}} \in
\im(L_{i,\der} \rightarrow \widetilde{L}^{\mathrm{u}}_i)$
that is conjugate self-dual and in general position,
and whose image $\widetilde v_i \in \widetilde{L}_i$
is also in general position.
Since
$\im(L_{i,\der} \longrightarrow \widetilde{L}_i)$
is contained in
$\im(L \longrightarrow \widetilde{L}_i)$,
we have that
$\widetilde{v}_{i} \in \im(L\longrightarrow\widetilde{L}_{i})$.

Thus, for suitably large $q$, $\widetilde{v}_{i}\in\im(L)$
is conjugate self-dual and in general position.
As usual, let $\widetilde{C}_i$ denote the group generated by
the Weyl orbit of $\widetilde{v}_i$.
Then the pre-image of $\widetilde{C}_i$ in $L$ contains a subgroup
$C_i$ isomorphic to $\widetilde{C}_i$,
where the isomorphism is equivariant under the Weyl group action.
Let $v_i\in L$ denote the inverse image of $\widetilde{v}_i$ under
this map isomorphism.
Then $v_i$ is conjugate self-dual and in general position.

Write $v=\sum_{i\in I}v_{i}$.
We had to assume 
Hypothesis \ref{hyp:SU-restrictions}(b)
in order to construct $v$ and it is easily seen to be 
conjugate self-dual.
Thus, under these conditions,
$G(\ff)$ has an irreducible, self-dual, Deligne-Lusztig, cuspidal representation.
\end{proof}

\section{Reductive groups}

Let $G$ be a connected reductive group defined over a finite
field $\ff$ of cardinality $q$.
Let $Z^{\circ}$ denote the identity component of the center $Z$
of $G$. 

\begin{lem}
\label{lem:ss}
Let $\tau'$ be a representation of $(G/Z^\circ)(\ff)$.
Then $\tau'$ pulls back to a representation $\tau$ of $G(\ff)$.
Moreover, if $\tau'$ satisfies any combination of the properties of being
irreducible, cuspidal, Deligne-Lusztig type or self-dual, 
then for the same combination of properties, so does $\tau$.
\end{lem}

\begin{proof}
We have a short exact sequence
\[
1\rightarrow Z^{\circ}\rightarrow G\rightarrow G/Z^{\circ}\rightarrow1.
\]
This gives a long exact sequence
\[
1
\rightarrow Z^{\circ}(\ff)
\rightarrow G(\ff)
\rightarrow G/Z^{\circ}(\ff)
\rightarrow \mathrm{H}^{1}(\ff,Z^{\circ}).
\]
By Lang's theorem, $\mathrm{H}^{1}(\ff,Z^\circ)$ is trivial.
Therefore, $G(\ff)$ surjects onto $(G/Z^{\circ})(\ff)$,
and thus $\tau'$ 
can be pulled back
to a representation $\tau$ of $G(\ff)$.
It is easy to see that if $\tau'$ satisfies any combination of the properties of
being irreducible, cuspidal, of Deligne-Lusztig type,
or self-dual,
then for the same combination of properties, the same is true for $\tau$.
\end{proof}

\begin{thm}
\label{thm:existence-cuspidal-finite}
Let $G$ be a connected reductive group over a finite field $\ff$.
Then $G(\ff)$ admits irreducible cuspidal representations. 
\end{thm}
\begin{proof}
By Lemma \ref{lem:ss}, we can assume that $G$ is semisimple. Then $G$
is isogenous to a group $\prod\Rest_{E_{i}/\ff}\widetilde{G}_{i}$
where the factors $\widetilde{G}_{i}$ are absolutely almost simple.
By an argument
as in the third paragraph of the proof of Lemma \ref{lem:odd-iso},
we can therefore
assume that $G$ is the restriction of scalars of an absolutely almost
simple group. The result then follows from Theorem \ref{thm:sc}. 
\end{proof}

\begin{thm}
\label{thm:existence-sd-finite}
Let $G$ be a connected reductive group defined over a finite
field $\ff$.
If $G$ satisfies Hypothesis \ref{hyp:SU-restrictions}(a),
then $G(\ff)$ admits irreducible, cuspidal, Deligne-Lusztig representations. 
Moreover, if $G$ also satisfies
Hypothesis \ref{hyp:SU-restrictions}(b) and has no factor of type
$A_n$ ($n$ even), 
then $G(\ff)$ admits irreducible, self-dual,
Deligne-Lusztig cuspidal representations.
If $G$ has a factor of type $A_n$ for some even $n$,
then $G(\ff)$ has no self-dual cuspidal representations.
\end{thm}

\begin{proof}
From Lemma \ref{lem:ss} and Proposition \ref{prop:semisimple},
$G(\ff)$ has irreducible,
cuspidal, Deligne-Lusztig representations,
and it has irreducible,
\emph{self-dual}, cuspidal, Deligne-Lusztig representations
if $G$ has no factor of type
$A_n$ ($n$ even).

Suppose $G$ has a factor of type $A_n$ for some even $n$.
Then there is a connected reductive $\ff$-group $H$,
and a central $\ff$-isogeny
$\SL_{n+1} \times H\longrightarrow G$
whose kernel has odd cardinality
and trivial intersection with $H$.
Theorem \ref{thm:sc} shows that
$\SL_{n+1}(\ff) \times H(\ff)$
has no self-dual cuspidal representations.
By Lemma \ref{lem:odd-iso}, neither does $G(\ff)$.
\end{proof}

\section{Reductive $p$-adic groups}

Let $F$ denote a non-archimedean local field,
with residue field $\ff$ of characteristic $p$ and order $q$.
Let $G$ be a connected reductive $F$-group.
For any point $x$ in the building of $G$ over $F$,
let $G(F)_x$, $G(F)_{x,0}$, and $\sfG_x(\ff)$
denote the stabilizer of $x$ in $G(F)$,
the parahoric subgroup of $G(F)$ associated to $x$,
and the reductive quotient of the parahoric subgroup,
i.e., the quotient of $G(F)_{x,0}$ by its pro-$p$-radical
$G(F)_{x,0+}$.
In particular, $\sfG_x$ is a connected reductive $\ff$-group.
When $G$ is a torus, all of the above are independent of the choice
of point $x$,
and it is customary to write $G(F)\bdd$, $G(F)_0$, and $G(F)_{0+}$
in place of $G(F)_x$, $G(F)_{x,0}$, and $G(F)_{x,0+}$.
Here, $G(F)\bdd$ is the maximal bounded subgroup of $G(F)$.

Let $\pi$ be an irreducible, supercuspidal representation of depth
zero. Then there exists a vertex $x$ in the building
of $G$ such that the restriction $\pi|G(F)_{x,0}$ contains the inflation
to $G(F)_{x,0}$ of an irreducible cuspidal representation $\kappa$
of $\sfG_x(\ff)$.

\begin{defn}[\cite{kaletha:regular-sc}*{Definition 3.4.19}]\label{def:reg-dep0}
The representation $\pi$ of $G(F)$ is called \emph{regular} if $\kappa$
is a Deligne-Lusztig cuspidal representation $\pm R_{\sf{S}^{\prime},\bar{\theta}}$ 
of $\sfG_x(\ff)$, associated to an elliptic maximal torus $\sf{S}'$ of $\sfG_x(\ff)$
and a character $\bar{\theta}:\sf{S}'(\ff)\rightarrow\mathbb{C}^{\times}$ that 
is \emph{regular} (\cite{kaletha:regular-sc}*{Definition 3.4.16}).
\end{defn}
Regular depth-zero supercuspidals are in bijection with $G(F)$-conjugacy
classes of \emph{elliptic regular pairs} $(S,\theta)$,
i.e., pairs in which $S$ is a maximally unramified elliptic maximal 
torus of $G$ and $\theta:S(F)\rightarrow \mathbb{C}^\times$ is a regular 
depth-zero character. We denote by $\pi(S,\theta)$ the regular depth-zero 
supercupidal representation of $G(F)$ associated to elliptic regular pair
$(S,\theta)$
as in \cite{kaletha:regular-sc}*{\S3.4.3}.

\begin{prop}
\label{prop:existence-sc-depth0}
The group $G(F)$ has depth-zero supercuspidal representations.
\end{prop}

As remarked in \S\ref{sec:intro}, our method
of proof is not new.

\begin{proof}
Let $x$ be a point in the building of $G$ over $F$ whose image
in the reduced building is a vertex.
By Theorem \ref{thm:existence-cuspidal-finite}, 
$\sfG_{x}(\ff)$ admits an irreducible, cuspidal representation
$\overline\rho$.
Let $\rho$ denote the inflation of $\overline\rho$ to $G(F)_{x,0}$.
Let $\tau$ denote any irreducible representation of $G(F)_x$
whose restriction to $G(F)_{x,0}$ contains $\rho$.
From \cite{moy-prasad:jacquet}*{Proposition 6.8},
the representation $\cInd_{G(F)_x}^{G(F)}\tau$ of $G(F)$
is irreducible and supercuspidal, and has depth zero.
\end{proof}

We now turn our attention to \emph{self-dual} supercuspidal representations,
starting with some situations where they do not exist.

\begin{lem}
\label{lemma:PGL-odd}
Suppose that $p$ is odd and
$G$ is an isotropic inner form of $\PGL_{n+1}$ for some even $n$.
Then $G(F)$ has no irreducible, self-dual, supercuspidal representations.
\end{lem}

\begin{proof}
There exists a
short exact sequence of connected $F$-groups
$$
1 \longrightarrow
\widetilde{Z} \longrightarrow
\widetilde{G} \longrightarrow
G \longrightarrow
1,
$$
where $\widetilde{G}$ is an inner form of $\GL_{n+1}$,
and $\widetilde{Z}$ is isomorphic to $\GL_1$.
Since $H^1(F,\widetilde{Z})$ is trivial,
the map $\widetilde{G}(F) \longrightarrow G(F)$ is thus surjective,
so it will be enough to show that $\widetilde{G}(F)$ has no self-dual
supercuspidal representations.
By \cite{dprasad:div-alg}*{Proposition 5},
a division algebra over $F$ of odd degree has no
irreducible, self-dual representations of dimension more than one.
The Jacquet-Langlands correspondence commutes with taking duals,
and every supercuspidal representation of $\widetilde{G}$
corresponds to a representation of a division algebra of dimension
more than one.
Therefore,
$\widetilde{G}(F)$ has no irreducible, self-dual, supercuspidal representations.
\end{proof}

\begin{prop}
\label{prop:no-self-dual}
Suppose that $p$ is odd and
some $F$-almost-simple factor of $G$ is isotropic,
of type $A_n$ for some even $n$, and an inner form of
a split group.
Then $G(F)$ has no self-dual supercuspidal representations.
\end{prop}

\begin{proof}
There exists a central $F$-isogeny
$G\longrightarrow H \times \Rest_{E/F}G_0$,
where $H$ is a connected reductive $F$-group,
$E/F$ is a finite, separable extension,
$G_0$ is an $E$-group that is an inner form of $\PGL_{n+1}$,
and the kernel of the isogeny has odd order.
By Lemma \ref{lem:odd-iso}, it will be enough to show that
$G_0(E)$ has no self-dual supercuspidal representations.
But this follows from Lemma \ref{lemma:PGL-odd}.
\end{proof}

We remark that if $G$ is an anisotropic group of type $A_n$,
then $G(F)$ does have self-dual supercuspidal representations
(e.g., the trivial representation), but they are not regular.

\begin{rem}
\label{rem:max-unram-torus}
Suppose that $G$ is quasi-split
over $F$.
As observed in \cite{kaletha:regular-sc}*{\S3.4}
or \cite{adler-fintzen-varma:kostant}*{\S2.4},
the building of $G$ over $F$ has a vertex $x$
that is ``absolutely special'',
in the sense that it is a special vertex in the building
of $G$ over $E$ for every algebraic extension $E/F$
of finite ramification degree.
Then the root systems of $G$ and $\sfG_x$ are isomorphic.
Let $\sfS$ denote a maximal elliptic $\ff$-torus in $\sfG_x$.
From \cite{kaletha:regular-sc}*{Lemma 3.4.3},
there is a maximally unramified
elliptic $F$-torus
$S$ in $G$ whose parahoric subgroup $S(F)_0$ is $S(F) \cap G(F)_{x,0}$,
and the image of $S(F)_0$ in $\sfG_x(\ff)$ is $\sfS(\ff)$.
We will be particularly interested in the case where $S$ satisfies
the following hypothesis.
\end{rem}

\begin{hyp}
\label{hyp:S-decomp}
Let $S(F)\odd$ be the largest subgroup of $S(F)\bdd$ that contains
$S(F)_0$ with odd index.
Then
$S(F)\odd/S(F)_{0+}$
is a direct factor of
$S(F)\bdd/S(F)_{0+}$.
\end{hyp}

\begin{rem}
\label{rem:hypothesis}
Hypothesis \ref{hyp:S-decomp}
is automatic if 
$S(F)_0/S(F)_{0+}$
is a direct factor of
$S(F)\bdd/S(F)_{0+}$.
Therefore, it is true
for all $S$ in $G$
in each of the following situations:
\begin{enumerate}[(i)]
\item
\label{item:unram}
$G$ splits over an unramified extension.
For in this case, $S(F)\bdd = S(F)_0$.
\item
\label{item:wild}
$G$ splits over a totally wild extension of an unramified extension.
For in this case, $S(F)\bdd/S(F)_0$ is a $p$-group,
and $S(F)_0/S(F)_{0+}$ has order prime to $p$.
\item
$G$ is simply connected.
For in this case,
$S(F)\bdd = S(F)\cap G(F)_x = S(F) \cap G(F)_{x,0} = S(F)_0$.
\item
$G$ is a unitary group.  If $G$ is unramified or wildly ramified,
then this follows from (\ref{item:unram}) or (\ref{item:wild}).
If $G$ is ramified and tame,
then it is easy to see the structure of $S$,
and so one can check directly.
\item
\label{item:p=2}
$p=2$.
In this case $S(F)_0 / S(F)_{0+}$ has odd order, since it is the group of rational points of an $\ff$-torus.
It follows that $S(F)\odd/S(F)_{0+}$ is the maximal subgroup of 
$S(F)\bdd/S(F)_{0+}$ of odd order and therefore the hypothesis is satisfied.
\end{enumerate}
\end{rem}

\begin{prop}
\label{prop:finite-to-local}
Suppose that $G$ is quasi-split over $F$.
Let $x$ be an absolutely special vertex in 
the building of $G$ over $F$.
If $\sfG_x(\ff)$ has
irreducible, cuspidal, Deligne-Lusztig representations,
then $G(F)$ has 
irreducible, depth-zero, supercuspidal regular representations.
Suppose $G$ satisfies Hypothesis \ref{hyp:S-decomp}.
If $\sfG_x(\ff)$ has such representations that are also
self-dual,
then so does $G(F)$.
\end{prop}

\begin{proof}
Let $\rho$ be an irreducible, Deligne-Lusztig, cuspidal representation
of $\sfG_x(\ff)$.
Then $\rho$ arises from a pair $(\sfS,\bar\theta)$,
where
$\sfS$ is a maximal elliptic $\ff$-torus in $\sfG_x$,
and $\bar\theta$ is a complex character of $\sfS(\ff)$
that is
in general position.
Let $S\subseteq G$ be a maximally unramified elliptic $F$-torus
as in Remark \ref{rem:max-unram-torus},
whose parahoric subgroup $S(F)_0$ is $S(F) \cap G(F)_{x,0}$,
and where the image of $S(F)_0$ in $\sfG_x(\ff)$ is $\sfS(\ff)$.
Inflate $\bar\theta$ to obtain a character of $S(F)_0$.
Choose an extension $\theta$ of this character to $S(F)$.
From \cite{kaletha:regular-sc}*{Lemmas 3.4.6 and 3.4.11},
$\theta|_{S(F)_0}$, and thus $\theta$, has trivial
stabilizer in $N(S,G)(F)/ S(F)$.
From \cite{kaletha:regular-sc}*{Lemma 3.4.18},
we obtain a regular, depth-zero, supercuspidal representation
$\pi_{(S,\theta)}$ of $G(F)$,
as desired.

Now suppose that $\rho$ is also self-dual,
and that $G$ satisfies Hypothesis \ref{hyp:S-decomp}.
Then we can choose $(\sfS,\bar\theta)$ so that $\bar\theta$
is in general position and conjugate self-dual.
In particular, $\bar\theta$ is conjugate to its inverse
via some element in $W(\sfG_x, \sfS)$, necessarily of order two.
Since $x$ is absolutely special,
by \cite{kaletha:regular-sc}*{Lemma 3.4.10(3)},
this implies that the inflation of $\bar\theta$ to $S(F)_0$
is conjugate to its inverse via an element $w$ of $W(G,S)$,
also of order two.
We have an odd number of ways of extending this character to a
character $\theta\odd$ on $S(F)\odd$, so we can and do choose
$\theta\odd$ so that it is conjugate to its inverse
via $w \in W(G,S)$.
From Hypothesis \ref{hyp:S-decomp},
we may extend $\bar\theta$ in a trivial way to obtain
a character of $S(F)\bdd$.
Since $S(F)$ is a direct product of $S(F)\bdd$ and an integer
lattice,
we may further extend our character in a trivial way
to a character $\theta$ of $S(F)$.
We have constructed $\theta$ to be conjugate to its inverse,
so the representation
$\pi_{(S,\theta)}$
is self-dual.
\end{proof}

\begin{prop}
\label{prop:finite-to-local-sc}
Suppose that $G$ is a simply connected $F$-group.
If the building of $G(F)$ has a vertex $x$ such that
$\sfG_x(\ff)$ has an irreducible, self-dual, cuspidal representation,
then $G(F)$ has an irreducible, self-dual, supercuspidal representation.
\end{prop}

\begin{proof}
Let $\rho$ be an irreducible, self-dual cuspidal representation
of $\sfG_x(\ff)$.
Inflate $\rho$ 
to the parahoric subgroup $G(F)_{x,0}$ of $G(F)$, and induce to $G(F)$.
From \cite{moy-prasad:jacquet}*{Proposition 6.8},
we obtain an irreducible, supercuspidal representation $\pi$.
Since $\rho$ is self-dual, so is $\pi$.
\end{proof}

\begin{prop}
\label{prop:gen-to-qsplit}
Let $G$ be a connected reductive $F$-group,
and let $G_0$ be its quasi-split inner form.
If $G_0(F)$ admits an irreducible, regular
(resp.\ self-dual regular) supercuspidal representation of
depth zero, then so does $G(F)$.
\end{prop}

\begin{proof}
Let $\pi$ be such a representation of $G_0(F)$.
Let $\pi\cong \pi_{(S_0,\theta_0)}$ for some maximally unramified
maximal $F$-elliptic torus $S_0\subset G_0$
and some depth-zero complex character $\theta_0$ of $S_0(F)$
that is in general position with respect to the action
of the Weyl group $W(G_0,S_0)(F)$
(and is conjugate self-dual if $\pi$ is assumed self-dual).

By \cite{kazhdan-varshavsky:depth-zero}*{Lemma 1.5.1},
there is a maximal elliptic torus
$S\subset G$ that is stably conjugate to $S_0$.
We thus have that $S$ and $S_0$ are $F$-isomorphic,
as are $W(G,S)$ and $W(G_0,S_0)$.
Therefore, $S(F)$ has a depth-zero complex character $\theta$
that is in general position with respect to the action
of the Weyl group $W(G,S)(F)$.
If $\theta_0$ is conjugate self-dual, then so is $\theta$.
\end{proof}

\begin{hyp}
\label{hyp:SU-restrictions-padic}
The group $G$ has no $F$-almost-simple factor
isogenous to the unitary group $\Rest_{E/F}\SU_{k+1}$, where
$E/F$ is totally ramified, 
and the unitary group is defined with respect
to an unramified quadratic extension of $F$, and
\begin{enumerate}[(a)]
\item
$k=2$ and $q=2$;
\item
$k=2$ and $q\in{\{3,4\}}$; or $k=3$ and $q\in{\{2,3,5\}}$;
or $k=4$ and $q\in{\{2,3, 4,5\}}$.
\end{enumerate}
\end{hyp}

\begin{thm}
\label{thm:self-dual-sc}
Let $G$ be a connected reductive $F$-group.
\begin{enumerate}[(a)]
\item
\label{item:exist-local-regular}
If $G$ satisfies Hypothesis \ref{hyp:SU-restrictions-padic}(a),
then
$G(F)$ has
irreducible, regular, supercuspidal representations of depth zero.
\item
\label{item:exist-local-regular-sd}
If $G$ also satisfies
Hypotheses
\ref{hyp:SU-restrictions-padic}(b)
and
\ref{hyp:S-decomp}
(the latter for all maximally unramified elliptic tori $S\subset G$),
and $G$ has no $F$-almost-simple factors of type
$A_n$ ($n$ even),
then
$G(F)$ has
irreducible, \emph{self-dual},
regular, supercuspidal representations of depth zero.
\end{enumerate}
\end{thm}

\begin{proof}
Let $G_0$ be the quasi-split inner form of $G$.
It is clear that $G_0$ satisfies the various parts
of Hypothesis \ref{hyp:SU-restrictions-padic}
if and only if $G$ does,
and the same goes for Hypothesis \ref{hyp:S-decomp}.
From Proposition \ref{prop:gen-to-qsplit}, we may replace $G$
by $G_0$, and assume from now on that $G$ is quasi-split.

Let $x$ be an absolutely special vertex in the building of $G(F)$.
Our result will follow from
Proposition \ref{prop:finite-to-local}
and Theorem \ref{thm:existence-sd-finite},
provided that 
we can show that $\sfG_x$ satisfies Hypothesis \ref{hyp:SU-restrictions}
and that $\sfG_x$ has a factor of type $A_n$ ($n$ even)
if and only if $G$ does.

The decomposition of $G$ into an almost-direct product of
a torus and $F$-almost-simple factors induces
an analogous decomposition of $\sfG_x$.

Suppose that $H$ is a factor of $G$,
and $\sfH_x$ is the corresponding factor of $\sfG_x$.
(Here we are identifying $x$ with its projection in the building of $H(F)$.)
Note that the connected reductive quotient of $(R_{E/F} H)(F)_{x,0}$ 
is the group of $\ff$-points of $\sfH_x$ if $E/F$ is totally ramified,
and of $R_{\ff_E/\ff} \sfH_x$ if $E/F$ is unramified
(and $\ff_E$ denotes the residue field of $E$).
Thus, we may assume that $H$ is absolutely almost simple.
If $H$ splits over an unramified extension, then $H$ and $\sfH_x$
have the same type
(e.g., $A_n$, $^2D_n$, etc.).
Suppose that $H$ splits only over a ramified extension.
From the proof of
\cite{haines-rostami:satake}*{Lemma 5.0.1},
the Weyl group of $\sfH_x$ over $\ff$ is isomorphic to the
relative Weyl group $W(H,T_0)$, where $T_0$ is a maximal $F$-split
torus in $H$.
In particular, $\sfH_x$ cannot be a simply laced group,
and so cannot have type $A_n$ or $^2A_n$.
\end{proof}

\begin{thm}
\label{thm:self-dual-sc-even}
Suppose that $G$ is a connected reductive $F$-group and $p=2$.
If $q=2$, then assume that $G$ has no factor of type $^2A_3$
or $^2A_4$.
Then $G(F)$ admits irreducible self-dual supercuspidal representations.
\end{thm}

\begin{rem}
At present, ``regular'' supercuspidal representations of positive depth
have not been defined
when $F$ has residual characteristic two.
Perhaps in the future they will be constructed from characters in general
position, as in the case of odd residual characteristic.
But even should that happen,
our proof will not be able to show that
all such groups admit \emph{regular} supercuspidals,
because of its reliance on 
Lemma \ref{lem:odd-iso}
and (when $q$ is small) on the existence of unipotent
cuspidal representations of $\SU(3)(\ff)$.
\end{rem}

\begin{proof}[Proof of Theorem \ref{thm:self-dual-sc-even}]
From Lemma \ref{lem:odd-iso}, we may replace $G$
by a direct product
$H\times G_0$, where
$G_0$ is a direct product of inner forms of groups of the form
$\Rest_{E/F}\SL_{n+1}$ ($n$ even),
for finite separable field extensions $E/F$;
and
$\Rest_{E/F}\SU_3$
for finite, separable, totally ramified field extensions $E/F$,
and the unitary groups are defined with respect to the quadratic
unramifed extension of $F$;
and no simple factor of $H$ has any of these types.
From Theorem \ref{thm:self-dual-sc}(\ref{item:exist-local-regular-sd}),
and Remark \ref{rem:hypothesis}(\ref{item:p=2}),
$H(F)$ has self-dual supercuspidal representations.
Therefore,
it will be enough to show that the same is true for inner forms
of
$\SL_{n+1}(E)$, and $\SU_3(E)$.

From
\cite{adler:self-contra}*{Theorem 6.1},
$\GL_{n+1}(E)$ has self-dual supercuspidal representations,
and since the restriction of such a representation to $\SL_{n+1}(E)$
decomposes into an odd number of summands,
at least one of them must be self-dual.
By the Jacquet-Langlands correspondence, 
the same is true for inner forms.

The groups $\SU_n$ (for $n$ odd) have no non-quasi-split inner forms.
To obtain self-dual supercuspidal representations of
$\SU_n(E)$, 
Proposition \ref{prop:finite-to-local-sc}
shows that it is enough to obtain
an irreducible, self-dual cuspidal representation of
$\SU_n(\ff_E)$,
where $\ff_E$ is the residue field of $E$.
Remark \ref{rem:SU-small} provides such a representation when
$n=3$.
\end{proof}

\begin{rem}
We have not determined whether or not $\SU(5)(\ff)$
has an irreducible self-dual cuspidal representation
when $\ff$ has order $2$.
If it does, then in Theorem \ref{thm:self-dual-sc-even},
we need not exclude groups containing a factor of type $^2A_4$
when $q=2$,
because we can deal with such factors in the same way that we
dealt with factors of type $^2A_2$,
changing only a few words of the proof.
\end{rem}

\end{document}